\documentclass{article}
\usepackage{hyperref}
\usepackage{amsmath, amssymb, amsthm, latexsym}

\newcommand{\Ric}{\operatorname{Ric}}
\newcommand{\hess}{\operatorname{Hess}}
\newtheorem{theo}{Theorem}[section]
\newtheorem{lem}[theo]{Lemma}

\begin{document}

\title{Painlev\'e analysis of the Bryant Soliton} %Title of paper

\author{Alejandro Betancourt}
%\email{betancourtde@maths.ox.ac.uk}
%\homepage[]{Your web page}
%\thanks{}
%\altaffiliation{}
%\affiliation{Mathematical Institute, University of Oxford, Oxford, OX2 6GG, UK}

\date{ }

%\pacs{}% insert suggested PACS numbers in braces on next line

\maketitle %\maketitle must follow title, authors, abstract and \pacs
\begin{abstract}
We carry out a Painlev\'e analysis of the systems of differential equations corresponding to the steady and the expanding, rotationally symmetric, gradient Ricci solitons on $\mathbb{R}^n$. For the steady case, dimensions of the form $n=k^2+1$ are singled out, with dimensions 2, 5, and 10 being particularly distinguished. Only dimension 2 is singled out for the expanding soliton.
\end{abstract}

\section{Introduction}
A Riemannian manifold $(M,g)$ is called a Ricci soliton if there exist a vector field $X$ on $M$ and $\varepsilon \in \mathbb{R}$ such that 
\begin{equation}\label{soliton}
 \Ric (g)+\frac{1}{2}\mathcal{L}_{X} g+ \frac{\varepsilon}{2} g=0,
\end{equation}
where $\Ric (g)$ is the Ricci curvature of $g$. The soliton is called shrinking, steady, or expanding depending on whether $\varepsilon$ is negative, zero, or positive. If the underlying vector field $X$ is the gradient of some smooth function $f \in C^\infty (M)$, equation \eqref{soliton} becomes
\begin{equation}\label{gradsoliton}
 \Ric (g)+ \hess (f) +  \frac{\varepsilon}{2} g=0,
\end{equation}
where $\hess (f)$ is the Hessian of $f$ with respect to the metric $g$. In this case we say the soliton is \textit{gradient} and we call $f$ the \textit{soliton potential}. Ricci solitons were introduced by Hamilton in \cite{hamilton1988} and they are natural generalizations of Einstein metrics (for this reason they are also called \textit{quasi-Einstein} metrics in physics literature). The study of Ricci solitons has received special attention in recent years because they naturally correspond to self similar solutions of the Ricci flow equation \cite{hamilton1982} and they often appear as long time solutions of this flow.  Also, Ricci solitons arise as singularity models for the Ricci flow \cite{hamilton1995compactness}.

Hamilton and Witten independently proved\cite{hamilton1982,witten1991} that there exists a (unique up to rescaling) rotationally symmetric, complete, steady gradient Ricci soliton with nontrivial potential on $\mathbb{R}^2$. This soliton is known as the cigar soliton.
In unpublished work, Robert Bryant generalized this construction and showed a similar result for $\mathbb{R}^3$. More generally, Bryant's arguments can be used to produce rotationally symmetric, complete, gradient steady solitons on $\mathbb{R}^n$ for every $n\geq 2$. These solitons are referred to in the literature as Bryant solitons. Following Bryant's ideas, it is also possible to show the existence of a one parameter family of rotationally symmetric, complete, expanding gradient solitons on $\mathbb{R}^n$ for every $n\geq 2$ (see Chapter 1.5 of \cite{chowvol2}).

Although the existence of both the steady and the expanding solitons above has been established for every $n\geq 2$, explicit expressions have been found just for $n=2$. These correspond to the cigar soliton in the steady case and, to the one parameter family of expanding solitons in \cite{gutperle03} (see also Chapter 2.4 of \cite{chowvol1}). By following the procedure described in sections 1.4 and 1.5 of Ref. \cite{chowvol2}, the Ricci soliton equation \eqref{gradsoliton} for both the Bryant and the expanding solitons on $\mathbb{R}^{n+1}$ can be transformed into systems of ODE's. The resulting ODE for the steady case is
\begin{equation} \label{steady}
\begin{aligned}
	\dot{x}&=x^2-xy+n-1 \\
	\dot{y}&=xy-nx^2,
\end{aligned}
\end{equation}
whereas the ODE for the expanding case is.
\begin{equation}\label{expand}
	\begin{aligned} 
	\dot{x}&=x^2-xy+\lambda z^2 +n-1  \\
	\dot{y}&=xy-nx^2+\lambda z^2 \\
	\dot{z}&=xz.
	\end{aligned}
\end{equation}
Here $x$, $y$, and $z$ are functions of a variable $t$ defined on an interval $t\in (-\infty, T)$ with $T\leq \infty$, and $\lambda=\varepsilon/2$ is constant. Our main objective is to carry out a Painlev\'e analysis of these systems to identify the values of $n$ for which these systems of equations may be integrable. We expect to find explicit expressions for the Bryant and the rotationally symmetric, expanding gradient, solitons in the dimensions singled out by this analysis.

Before we go on, we briefly recall that the Painlev\'e test is a procedure, first pioneered by Sofia Kowaleski \cite{sofia}, that helps us determine the integrability of a given system of ODE's. The idea behind this method is that integrability of a system is usually associated to the existence of a large family of meromorphic (or meromorphic is some rational power of the variable) solutions with moveable singularities. The Painlev\'e test identifies such families by carrying out the following steps:
\begin{enumerate}
	\item Assume that there exists a meromorphic solution of the system (taking $t$ to be complex) and determine the\textit{ leading order} terms.
	\item Find a recursion formula for the series solution and compute the \textit{resonances}, that is, the steps of the expansion at which free parameters may enter.
	\item Check the \textit{compatibility conditions} at each resonance to verify that the recursion relation can be solved at every step.
	\item Verify that the series converges in a punctured neighborhood of the singularity.
\end{enumerate}
We say that a system passes the \textit{strong} Painlev\'e test if it is possible to find a meromorphic solution that depends on a maximal number of free parameters. A system passes the \textit{weak} Painlev\'e test if there exists a meromorphic expansion in some rational power of the variable $t$ with a full set of free parameters. Details on the Painlev\'e test can be found in Chapter 8 of Ref. \cite{tabor}.

\section{Analysis of the steady soliton} \label{steadysol}
We begin with the Painlev\'e analysis of system \eqref{steady} corresponding to the Bryant soliton in $\mathbb{R}^{n+1}$.
\subsection{Leading order analysis} \label{first}
Take expansions of the form $x=at^\alpha+\ldots$, and $y=bt^\beta+\ldots$ to compute the leading order terms of the system. Substitution in \eqref{steady} yields the relation
\begin{subequations} \label{coefsteady}
\begin{align} 
	a \alpha t^{\alpha -1} &[=] a^2 t^{2\alpha} -ab t^{\alpha + \beta}\label{leadorunosteady} \\
	b \beta t^{\beta -1} &[=] ab t^{\alpha + \beta} -n a^2 t^{2 \alpha},\label{leadordossteady}
\end{align}
\end{subequations}
where the notation $[=]$ is employed to denote that only the leading order terms on each side are equal (for example, if we had $\alpha<\beta$, then  \eqref{leadorunosteady} says that only $a \alpha t^{\alpha -1}$ and $a^2 t^{2\alpha}$ are equal). To find the leading exponents first assume that $\alpha < \beta$. Equation \eqref{leadorunosteady} implies that $\alpha=-1$ and \eqref{leadordossteady} then yields $\beta=-1$, a contradiction. Conversely, if $\beta < \alpha$ equation \eqref{leadordossteady} readily yields $\alpha =-1$, which in turn implies $\beta= -1$, a contradiction. Therefore, the only admissible leading exponents are $\alpha= \beta=-1$. From this we get that the coefficients must satisfy
\begin{equation} \label{estedi}
 \begin{aligned}
-a=& a^2-ab\\
-b=& ab-na^2.
\end{aligned} 
\end{equation}
This system has two nonzero solutions when $n>1$, namely $a_+=-1/(\sqrt n +1), \,b_+=\sqrt n/(\sqrt{n}+1)$; and $a_-=1/(\sqrt n-1), \, b_-=\sqrt n/(\sqrt{n}-1)$. For $n=1$ the only nonzero solution is $a=-1/2$, $b=1/2$. Both solutions must be analysed independently.
\subsection{Recursion formula and resonances}
We now make the ansatz 
\begin{equation*}
x=\sum_{i=0}^\infty a_i t^{-1+i}, \qquad y=\sum_{i=0}^\infty b_i t^{-1+i},
\end{equation*}
where $a_0, \, b_0$ are chosen to be either $a_+, \, b_+$ or $a_-, \, b_-$ from the previous section. Substitution in \eqref{steady} yields the recursion relation 
\begin{equation} \label{recrel}
\begin{aligned}
\left( \begin{array}{cc}
i-1-2a_0+b_0 & a_0 \\
2na_0-b_0 & i-1-a_0
\end{array} \right)
 \left( \! \! \begin{array}{c}
a_i\\
b_i
\end{array} \! \! \right)=\\
 \left( \begin{array}{c}
 \sum_{k=1}^{i-1} a_k(a_{i-k}-b_{i-k})\\
 \sum_{k=1}^{i-1} a_k(b_{i-k}-na_{i-k})
\end{array} \right)
+ \delta_{2,i} \binom{n-1}{0},
\end{aligned}
\end{equation}
where $\delta_{2,i}$ is a Kronecker delta. Notice that the second term on the right hand side is added to account for the $n-1$ in the first equation of \eqref{steady}. We will denote the $2 \times 2$ matrix on the left hand side by $X(i)$. The coefficients $a_i, \,b_i$ are uniquely determined for every value of $i$ except when $X(i)$ is singular. At these steps we have resonances and thus free parameters may enter the expansion. Solving $\det X(i)=0$ we obtain the following resonances:
\begin{enumerate}
	\item If we take $a_0, \, b_0$ to be $a_+, \, b_+$, then resonances occur at $i=-1$ and $i=2\sqrt{n}/(\sqrt{n}+1)$.
	\item If we take $a_0, \, b_0$ to be $a_-, \, b_-$, we have resonances at $i=-1$ and $i=2\sqrt{n}/(\sqrt{n}-1)$.
\end{enumerate}
The resonance at $i=-1$ corresponds to the fact that the position of the singularity is a free parameter of the system (we have chosen it to be $t=0$ in our ansatz to make computations easier). First we will find the values of $n$ for which system \eqref{steady} passes the strong Painlev\'e test, i.e., we are interested in the cases where the second resonance is an integer.
\begin{lem} \label{integerres}
The second resonance of system \eqref{steady} is an integer if and only if $n=1,\, 4$, or $9$.
\end{lem}
\begin{proof}
The quantity $i=2\sqrt{n}/(\sqrt{n}+1)$ is an integer if an only if $n=1$. To see this consider the function $f(x)=2x/(x+1)$. This function takes only positive values when $x\geq 0$ and is strictly increasing on this region. Furthermore, $f(1)=1$, and $f(x) \rightarrow 2$ from below as $x \rightarrow	\infty$, so it can not take any other integer values.

On the other hand, the quantity $i=2\sqrt{n}/(\sqrt{n}-1)$ is an integer if and only if $n=4$ or $n=9$. Indeed, consider the function $f(x)=2x/(x-1)$. It takes positive values if $x>1$ and is strictly decreasing in this region. Furthermore, $f(2)=4$, $f(3)=3$, and $f(x)\rightarrow 2$ asymptotically from above, so it does not take any other integer values.
\end{proof}
The following table summarizes the values that the second resonance takes depending on the value of $n$ and on our choice of initial $a_0,\,b_0$.
\begin{table}[h!]
\caption{Resonances for $n=1,\, 4$, and $9$.}
\centering
  \begin{tabular}{cccccccc}
    \hline
		$n \quad$ & $a_+$ & $b_+$ & $i$ & $\quad$ & $a_-\quad$ & $b_-\quad$ & $i\quad$ \\ \hline
		1$\quad$ & $-1/2 \quad$ & $1/2 \quad$ & 1 & $\quad$ & $\quad$ & $\quad$ & $\quad$ \\ 
		4$\quad$  & $-1/3\quad$ & $2/3 \quad$ & $4/3$ & $\quad$ & $1\quad$ & $2\quad$ & $4\quad$  \\
		9$\quad$  & $-1/4\quad$ & $3/4 \quad$ & $3/2$ & $\quad$ & $1/2\quad$  & $3/2\quad $ & $3\quad$ \\ \hline
  \end{tabular}
\end{table}

We are also interested in the weak Painlev\'e test, so we have to consider the cases where $i$ takes rational values as well. In this case it is obvious that the second resonance is a rational number if and only $n$ is a perfect square.

\subsection{Compatibility conditions} 
\subsubsection{Compatibility conditions for meromorphic expansions}
We proceed to verify the compatibility conditions for the three cases with integer resonances (in these cases our original ansatz corresponds to the series expansion of a meromorphic function). We start with the case $n=1$. In this case the recursion relation \eqref{recrel} becomes
\begin{equation*}
	\left( \begin{array}{cc}
	i+1/2 & -1/2 \\
	-3/2 & i-1/2
	\end{array} \right)
	\left( \! \! \begin{array}{c}
	a_i\\
	b_i
	\end{array} \! \!  \right)=\\
	\left( \begin{array}{c}
 	\sum_{k=1}^{i-1} a_k(a_{i-k}-b_{i-k})\\
 	\sum_{k=1}^{i-1} a_k(b_{i-k}-a_{i-k})
	\end{array} \right).
\end{equation*}
Thus, for $i=1$ we have
\begin{equation*}
	\frac{1}{2}\left( \begin{array}{rr}
	3 & -1 \\[0.3em]
	-3 & 1
	\end{array} \right)
	\left( \! \! \begin{array}{c}
	a_1\\
	b_1
	\end{array} \! \!  \right)=\\
	\left( \! \! \begin{array}{c}
 	0 \\
 	0
	\end{array} \! \! \right).
\end{equation*}
This equation has infinitely many solutions, so free parameters enter the expansion. Compatibility conditions are satisfied trivially.

For the case $n=4$ we have $a_0=1$, $b_0=2$ and hence the recursion relation is given by
\begin{equation*} \label{res4}
	\left( \begin{array}{cc}
	i-1 & 1 \\
	6 & i-2
	\end{array} \right)
	\left( \! \! \begin{array}{c}
	a_i\\
	b_i
	\end{array} \! \!  \right)=\\
	\left( \begin{array}{c}
 	\sum_{k=1}^{i-1} a_k(a_{i-k}-b_{i-k})\\
 	\sum_{k=1}^{i-1} a_k(b_{i-k}-4a_{i-k})
	\end{array} \right)+\delta_{2,i}\binom{3}{0}.
\end{equation*}
For the resonance $i=4$ we obtain the equation
\begin{equation*}
	\left( \begin{array}{cc}
	3 & 1 \\
	6 & 2
	\end{array} \right)
	\left( \! \! \begin{array}{c}
	a_4\\
	b_4
	\end{array} \! \!  \right)=\\
	\left( \! \! \begin{array}{c}
 	0 \\
 	0
	\end{array} \! \! \right),
\end{equation*}
which shows that compatibility conditions also hold in this case.

Finally, for the case $n=9$ we have $a_0=1/2$ and $b_0= 3/2$. The recursion relation is 
\begin{equation*} 
	\frac{1}{2}\left( \begin{array}{cc}
	2i-1 & 1 \\
	15 & 2i-3
	\end{array} \right)
	\left( \! \! \begin{array}{c}
	a_i\\
	b_i
	\end{array} \! \!  \right)=\\
	\left( \begin{array}{c}
 	\sum_{k=1}^{i-1} a_k(a_{i-k}-b_{i-k})\\
 	\sum_{k=1}^{i-1} a_k(b_{i-k}-9a_{i-k})
	\end{array} \right)+\delta_{2,i}\binom{8}{0}.
\end{equation*}
At the resonance $i=3$ we have 
\begin{equation*}
	\frac{1}{2}\left( \begin{array}{cc}
	5 & 1 \\
	15 & 3
	\end{array} \right)
	\left( \! \! \begin{array}{c}
	a_3\\
	b_3
	\end{array} \! \!  \right)=\\
	\left( \! \! \begin{array}{c}
 	0 \\
 	0
	\end{array} \! \! \right),
\end{equation*}
and hence free parameters enter the expansion. The majorisation argument from Section 6 of Ref. \cite{dancer2001jgp} further guarantees that the series obtained by these recursion relations converge in some punctured neighborhood of the singularity. This implies the following theorem.
\begin{theo}
System \eqref{steady}, corresponding to the Bryant soliton on $\mathbb{R}^{n+1}$, passes the strong Painlev\'e test if and only if $n=1,\,4$, or $9$.
\end{theo}
\subsubsection{Compatibility conditions for expansions with branch points}
We now turn out attention to the case where $n>1$ is a perfect square. For this case we make an ansatz of the form
\begin{equation*} 
 x=\sum_{i=0}^\infty a_i t^{-1+i/Q} \quad \text{ and }\quad y=\sum_{i=0}^\infty b_i t^{-1+i/Q},
\end{equation*}
where $Q:=\sqrt{n}+1$ if we choose $a_+,\,b_+$ as leading coefficients, or $Q:=\sqrt{n}-1$ if we choose $a_-,\,b_-$ (this corresponds to a series expansion of a function with a branch point). For this ansatz equation \eqref{steady} yields the recursion relation
\begin{equation} \label{recrelrat}
\begin{aligned}
\left( \begin{array}{cc}
i/Q-1-2a_0+b_0 & a_0 \\
2na_0-b_0  & i/Q-1-a_0
\end{array} \right)
\left( \! \! \begin{array}{c}
a_i\\
b_i
\end{array} \! \! \right)=\\
\left( \begin{array}{c}
 \sum_{k=1}^{i-1} a_k(a_{i-k}-b_{i-k})\\
 \sum_{k=1}^{i-1} a_k(b_{i-k}-na_{i-k})
\end{array} \right)+\delta_{2Q,i}\binom{n-1}{0}.
\end{aligned}
\end{equation}
We will keep denoting the matrix on the left hand side by $X(i)$. Our analysis from the previous section shows that we have a resonance when $i=2\sqrt{n}$ regardless of the choice of leading coefficients. To check that compatibility conditions at the resonance hold it is enough to show that the right hand side of \eqref{recrelrat} vanishes. 
\begin{lem}
For every perfect square $n>1$, the right hand side of \eqref{recrelrat} vanishes at step $i=2\sqrt{n}$. 
\end{lem}
\begin{proof}
We proceed by cases. First assume that we choose $a_+, \,b_+$ as the leading coefficients (recall that $Q=\sqrt{n}+1$ in this case). The right hand side of \eqref{recrelrat} vanishes at $i=0$. An inductive argument further shows that if $a_k=b_k=0$ for all $0<k<i$ for a fixed $i<2\sqrt n \, $, then $a_{i}=b_{i}=0$. Indeed, note that in this cases $X(i)$ is nonsingular and the right hand side of \eqref{recrelrat} vanishes. This shows that $a_i=b_i=0$ for all $i<2\sqrt{n}$. Therefore, at the resonance $i=2\sqrt{n}$ all the terms in the sum in the right hand side of \eqref{recrelrat} vanish and also the Kronecker delta is zero, since $2\sqrt{n}<2Q$. Hence, the right hand side of \eqref{recrelrat} vanishes at the resonance as claimed.

If we choose $a_-, \,b_-$ as the leading order coefficients the argument is essentially the same, although we have to be more careful because this time $Q=\sqrt{n}-1$ and we might get some nonzero coefficients before the resonance. The inductive argument above shows that if $i<2Q$, then $a_i=b_i=0$. At step $i=2Q$ the matrix $X(i)$ is nonsingular and the right hand side of \eqref{recrelrat} is not zero, so at least one of the coefficients $a_{2Q},\, b_{2Q}$ is not zero. The next step of the recursion is
\begin{equation*}
X \left( 2\sqrt n-1\right) \left( \! \! \begin{array}{c}
a_{2\sqrt n-1}\\
b_{2\sqrt n-1}
\end{array} \! \! \right)=\binom{a_1(a_{2Q}-b_{2Q})+a_{2Q}(a_1-b_1)}{a_1(b_{2Q}-na_{2Q})+a_{2Q}(b_1-na_1)}=\binom{0}{0},
\end{equation*}
so even if we get nonzero coefficients at step $i=2Q$, the coefficients at the following step are both zero. At step $i=2\sqrt n$ we have
\begin{equation*}
X \left( 2\sqrt n\right) \left( \! \! \begin{array}{c}
a_{2\sqrt n}\\
b_{2\sqrt n}
\end{array} \! \! \right)=\binom{a_2(a_{2Q}-b_{2Q})+a_{2Q}(a_2-b_2)}{a_2(b_{2Q}-na_{2Q})+a_{2Q}(b_2-na_2)}.
\end{equation*}
From this we can easily see that the right hand side is again zero, which proves the claim.
\end{proof}
Using the majorisation argument from Section 6 of Ref. \cite{dancer2001jgp} again to prove convergence of our series, we obtain the following result.
\begin{theo}
System \eqref{steady}, corresponding to the Bryant soliton on $\mathbb{R}^{n+1}$, passes the weak Painlev\'e test for every perfect square $n$.
\end{theo}

\section{Analysis of the expanding soliton}
We now turn our attention to system \eqref{expand} corresponding to the rotationally symmetric, expanding soliton on $\mathbb{R}^{n+1}$.
\subsection{Leading order exponents}
We take $x=at^\alpha+\ldots$, $y=bt^\beta+\ldots$, and $z=ct^\gamma+\ldots$ to compute the leading order coefficients. Substitution in \eqref{expand} yields the relations
\begin{subequations}\label{coefexpand}
\begin{align}
	a \alpha t^{\alpha -1} &[=] a^2 t^{2\alpha} -ab t^{\alpha + \beta} +\lambda c^2 t^{2 \gamma}\label{leaduno}\\
	b \beta t^{\beta -1} &[=] ab t^{\alpha + \beta} -n a^2 t^{2 \alpha}+\lambda c^2 t^{2 \gamma}\label{leadudos}\\
	c \gamma t^{\gamma -1} &[=] ac t^{\alpha + \gamma}.\label{leadtres}
\end{align}
\end{subequations}
From \eqref{leadtres} we immediately deduce that $\alpha=-1$. We must also have that both $\beta \geq -1$ and $\gamma \geq -1$. To see this first suppose that exactly one of the two exponents is less than $-1$. It is easy to see that \eqref{leaduno} leads to a contradiction. If both $\beta<-1$ and $\gamma<-1$ we must have that $\beta-1=2\gamma$ and $ab=\lambda c^2$, otherwise \eqref{leaduno} could not hold. Equation \eqref{leadudos} would then imply that $\beta =2a$. On the other hand, \eqref{leadtres} implies $\gamma=a$, which means $2\gamma=\beta$, a contradiction. Thus, $\beta \geq -1$ and $\gamma \geq -1$. Once that we know this, we can use \eqref{leadudos} to see that actually $\beta=-1$, otherwise the terms $-n a^2 t^{2 \alpha}$ and $\lambda c^2 t^{2 \gamma}$ would have to cancel out, which leads to a contradiction. We cannot find a value for $\gamma$ by only looking at the exponents, so we distinguish two cases, both of which must be analysed: $\alpha = \beta = \gamma =-1$, and $\alpha = \beta =-1$, $\gamma> -1$. We present the rest of the analysis for each of the two cases separately.
\subsection{Expansion with equal leading exponents} \label{equals}
We first study the case in which all exponents are $-1$. Equation \eqref{coefexpand} implies that the leading coefficients are related by 
	\begin{align*}
		-a=& a^2-ab+\lambda c^2\\
		-b=& ab-na^2+\lambda c^2\\
		-c=&ac.
	\end{align*} 
The only nontrivial solutions of this system are $a=-1$, $b=-n$, and $c=\pm \sqrt{n/\lambda}$. For the moment we will only consider the positive value of $c$ (the rest of the procedure is completely analogous for the negative value). We now make the ansatz
\begin{equation*}
 x=\sum_{i=0}^\infty a_i t^{-1+i},\quad y=\sum_{i=0}^\infty b_i t^{-1+i}, \textnormal{ and} \quad z=\sum_{i=0}^\infty c_i t^{-1+i},
\end{equation*}
where $a_0=a$, $b_0=b$, and $c_0=c$ are the coefficients just calculated. Substituting in equation \eqref{expand} we get the recursion relation
\begin{equation*}
\begin{aligned}
		\left( \begin{array}{ccc}
			i-1-2a_0+b_0 & a_0 & -2\lambda c_0\\
			2na_0-b_0 & i-1-a_0 & -2\lambda c_0\\
			-c_0 & 0 &i-1-a_0
		\end{array} \right) \left( \begin{array}{c}
			a_i\\
			b_i\\
			c_i		
		\end{array}\right)=\\
		\left( \begin{array}{c}
 \sum_{k=1}^{i-1} a_k(a_{i-k}-b_{i-k})+\lambda c_kc_{i-k}\\
 \sum_{k=1}^{i-1} a_k(b_{i-k}-na_{i-k})+\lambda c_kc_{i-k}\\
 \sum_{k=1}^{i-1}a_k c_{i-k} \end{array} \right)+ \delta_{2,i}
 \left( \begin{array}{c}
 n-1\\
 0\\
 0 \end{array}\right).
\end{aligned}
\end{equation*}
Just as in the steady case, we call the $3\times 3$ matrix on the left $X(i)$. The polynomial $\det X(i)$ has roots $i=-1$, $i=(n-\sqrt{n^2+8n})/2$, and $i=(n+\sqrt{n^2+8n})/2$. From this we see that resonances will be rational if and only if $n^2+8n$ is a perfect square.
\begin{lem}
Let $n$ be a positive integer. Then $n^2+8n$ is a perfect square if and only if $n=1$.
\end{lem}
\begin{proof}
Suppose $n^2+8n=m^2$ is a perfect square (we take $m$ to be nonnegative). Solving for $n$ in this expression we get $n=-4\pm \sqrt{16+m^2}$.
Since $n$ is an integer by hypothesis, $16+m^2$ must be a perfect square. This means that $(m,4,\sqrt{16+m^2})$ must be a Pythagorean triple, and therefore it has to be either $(0,4,4)$ or $(3,4,5)$. Hence, we must have that $m=0$ or $m=3$. If $n^2+8n=0$ then $n=-8$ or $n=0$. On the other hand $n^2+8n=9$ implies $n=-9$ or $n=1$.
\end{proof}
This result implies that resonances are rational only in the case $n=1$. In this case $\det X(i)$ has roots $i=-1$ and $i=2$ (the former with multiplicity 2). The leading coefficients are $a_0=-1$, $b_0=-1$, and $c_0=1/\sqrt \lambda$ and so we have the recursion relation
\begin{equation*}
		\left( \! \begin{array}{ccc}
			i & -1 & -2 \sqrt \lambda \\
			-1 & i & -2 \sqrt \lambda\\
			-1/\sqrt \lambda \, & 0 & i
		\end{array} \! \right) \! \! \left( \begin{array}{c}
			a_i\\
			b_i\\
			c_i		
		\end{array}\right) \! = \!
\left( \begin{array}{c}
 \sum_{k=1}^{i-1} a_k(a_{i-k}-b_{i-k})+\lambda c_kc_{i-k}\\
 \sum_{k=1}^{i-1} a_k(b_{i-k}-na_{i-k})+\lambda c_kc_{i-k}\\
 \sum_{k=1}^{i-1}a_k c_{i-k} \end{array} \right).
\end{equation*}
From this we see that $a_1=b_1=c_1=0$ and therefore the right hand side vanishes at the resonance $i=2$, so compatibility conditions are satisfied. Nonetheless, given that $\dim \ker X(2)=1$, our series expansion only has two free parameters (one being the position of the singularity and the other one entering at the resonance $i=2$) and is therefore not the general solution of \eqref{expand}. 

As we stated before, if instead we choose $c=-\sqrt{n/\lambda}$, the situation is essentially the same. In particular, $\det X(i)$ is the same polynomial regardless of the sign of $c$, which means that we will have the same resonances as before. Compatibility conditions also hold at the top resonance $i=2$, but not enough free parameters enter the expansion.

\subsection{Expansions with different leading coefficients}
Next we analyse the case in which the leading exponents are $\alpha=\beta=-1$, and $\gamma>-1$ is yet to be determined. In this case the leading coefficients are related by 
	\begin{align*}
		-a=& a^2-ab\\
		-b=& ab-na^2\\
		c\gamma =&ac.
	\end{align*} 
From this we immediately see that $\gamma = a$ and that $c\neq 0$ can be chosen arbitrarily (this means we have already found a free parameter). Also, $a$ and $b$ must be nonzero solutions of equation \eqref{estedi} corresponding to the steady soliton, so we will have to further consider two subcases. Notice that $a$ must be rational because apart from being a leading coefficient it is also a leading exponent, so we will require $n$ to be a perfect square.

We will consider series solutions of \eqref{expand} of the form
\begin{equation*}
  x=\sum_{i=0}^\infty a_i t^{-1+i/Q},\quad y=\sum_{i=0}^\infty b_i t^{-1+i/Q}, \textnormal{ and} \quad z=\sum_{i=0}^\infty c_i t^{a+i/Q}.
\end{equation*}
Here $a_0,\, b_0$ are taken to be either $a_+, \,b_+$, or $a_-, \,b_-$ from Section \ref{first}, and $c\neq 0$ is arbitrary. Just as before, in the first case we set $Q:=\sqrt{n}+1$, and in the second $Q:=\sqrt{n}-1$. To compute the recursion relation corresponding to these series, notice that in both cases $a_0=-1+b_0$. This means that the series expansion for $z$ is $\sqrt n$ steps ahead of the expansions for $x$ and $y$. Indeed, $a_i$ and $b_i$ are the coefficients of $t^{-1+i/Q}$ in $x$ and $y$, and the coefficient for the same power of $t$ in $z$ is $c_{i-\sqrt n}$. Thus, the recursion relation is given by
\begin{equation} \label{recrelexpdos}
	\begin{aligned}
		\left( \begin{array}{ccc}
			i/Q-1-2a_0+b_0 & a_0 & 0\\
			2na_0-b_0 & i/Q-1-a_0 & 0\\
			-c_0 & 0 &i/Q
		\end{array} \right) \left( \begin{array}{c}
			a_i\\
			b_i\\
			c_i		
		\end{array}\right)=\\
		\left( \begin{array}{c}
 \sum_{k=1}^{i-1} a_k(a_{i-k}-b_{i-k})+\lambda c_{k-\sqrt n}\,c_{i-k-\sqrt n}\\
 \sum_{k=1}^{i-1} a_k(b_{i-k}-na_{i-k})+\lambda c_{k-\sqrt n}\, c_{i-k-\sqrt n}\\
 \sum_{k=1}^{i-1}a_k c_{i-k} \end{array} \right)+ \delta_{2Q,i}
 \left( \begin{array}{c}
 n-1\\
 0\\
 0 \end{array}\right).
	\end{aligned}
\end{equation}
As before, we call the $3 \times 3$ matrix on the left $X(i)$. If we set $\nu:=i/Q$, we have that 
\begin{equation*}
\det X(i)=\nu (\nu-1)\left(\nu-\frac{2 \sqrt{n}}{Q}\right),
\end{equation*}
so we have resonances at $i=-Q$, $i=0$, and $i=2\sqrt n$. The first resonance corresponds to the arbitrariness in the position of the singularity and the second to the arbitrariness of $c_0$. To pass the weak Painlev\'e test we will need another free parameter entering the expansion at the third resonance. Unfortunately this is not the case, as the next result shows.
\begin{lem}
The right hand side of \eqref{recrelexpdos} is not in the image of $X(i)$ at the resonance $i=2\sqrt{n}$. In particular, compatibility conditions at the top resonance do not hold.
\end{lem}
\begin{proof}
We proceed by cases. First assume that $a_0,\,b_0$ are taken to be $a_+, \,b_+$ from Section \ref{first}. In this case $Q=\sqrt n+1$. Using an inductive argument as for the steady case, we see that if $a_k=b_k=c_k=0$ for all $0<k<i$ for a fixed $i<2\sqrt{n}$, then $a_i=b_i=c_i=0$. Indeed, for such $i$ the matrix $X(i)$ is nonsingular and the right hand side of \eqref{recrelexpdos} vanishes. Thus, $a_i=b_i=c_i=0$ for all $i<2\sqrt{n}$. When $i=2\sqrt{n}$, the right hand side of \eqref{recrelexpdos} is nonzero. The recursion relation at this step is
\begin{equation*}
 \frac{1}{Q}\left( \begin{array}{ccc}
			1+2\sqrt{n} & -1 & 0\\
			-2n-\sqrt{n} & \sqrt{n} & 0\\
			-c_0Q & 0 &2\sqrt{n}
		\end{array} \right) \left( \begin{array}{c}
			a_{2\sqrt{n}}\\
			b_{2\sqrt{n}}\\
			c_{2\sqrt{n}}		
		\end{array}\right)=\\
		\left( \begin{array}{c}
 \lambda c_0^2\\
  \lambda c_0^2\\
  0 \end{array} \right).
\end{equation*}
Clearly the right hand side is not in the image of $X(2\sqrt{n})$.

Conversely, if we set $a_0,\,b_0$ to be $a_-, \,b_-$, then we have that $Q=\sqrt n-1$ and therefore we get some nonzero coefficients before the resonance. Arguing as before, if $a_k=b_k=c_k=0$ for all $0<k<i$ for a fixed $i<2Q$, then $a_i=b_i=c_i=0$, so all the coefficients before step $i=2Q$ vanish. At $i=2Q$ we have the relation
\begin{equation*}
 X(2Q)\left( \begin{array}{c}
			a_{2Q}\\
			b_{2Q}\\
			c_{2Q}		
		\end{array}\right)=\left( \begin{array}{c}
			n-1\\
			0\\
			0		
		\end{array}\right),
\end{equation*}
so at least one of the coefficients at this step is not zero. Computing the next step of the recursion we see that the right hand side of \eqref{recrelexpdos} is zero, since the only possible nonzero coefficients are $c_0$, $a_{2Q}$, $b_{2Q}$, $c_{2Q}$, and all of these appear multiplied by zero. Thus, all coefficients at step $i=2\sqrt n-1$ are zero. At the resonance $i=2\sqrt n$ we obtain the relation
\begin{equation*}
 \frac{1}{Q}\left( \begin{array}{ccc}
			-1+2\sqrt{n} & 1 & 0\\
			2n-\sqrt{n} & \sqrt{n} & 0\\
			-c_0Q & 0 &2\sqrt{n}
		\end{array} \right) \left( \begin{array}{c}
			a_{2\sqrt{n}}\\
			b_{2\sqrt{n}}\\
			c_{2\sqrt{n}}		
		\end{array}\right)=\\
		\left( \begin{array}{c}
 \lambda c_0^2\\
  \lambda c_0^2\\
  0 \end{array} \right).
\end{equation*}
Again, the vector on the right hand side is not in the image of $X(2\sqrt{n})$.
\end{proof}
This lemma, together with the results from Section \ref{equals} imply the following theorem.
\begin{theo}
System \eqref{expand}, corresponding to the rotationally symmetric, expanding soliton in $\mathbb{R}^{n+1}$, does not pass the weak (and therefore the strong) Painlev\'e test for any value of $n$.
\end{theo}

\section{Concluding remarks}
We can summarize the results of the Painlev\'e analysis as follows: for the Bryant soliton we found that the cases $n=1,\, 4$ and $9$ pass the strong  Painlev\'e test, i.e. the general solution of the system is meromorphic with a moveable singularity. This means that the system is likely to be integrable in these cases. As we mentioned in the introduction, explicit solutions are known only for the case $n=1$. Our results suggest that we can also find similar expressions for the cases $n=4$ and $n=9$. We also found that the  Bryant soliton passes the weak Painlev\'e test for every perfect square $n$, that is, the general solution is meromorphic in some rational power of $t$. This suggests that it should be possible to give solutions with movable branch points explicitly. The explicit integration of the Bryant soliton equation is further discussed in \cite{bdw14}, under a new change of variables.

On the other hand, the expanding soliton does not pass the weak Painlev\'e test for any value of $n$. In the case $n=1$, however, compatibility conditions at the resonances are satisfied, which suggests that there is an integrable subsystem. This is consistent with the one existence of the one parameter family mentioned in the introduction. For $n>1$ though, the Painlev\'e test suggests that there are no explicit solutions of \eqref{expand}.

\section*{Acknowledgements}
The author wishes to acknowledge the Mexican Council of Science and Technology (CONACyT) and the Mathematical Institute of the University of Oxford for their support.

\end{document}